\documentclass[10pt]{amsart}
\usepackage[applemac]{inputenc}
\usepackage[english]{babel}

 \newtheorem{thm}{Theorem}[section]
 
 \newtheorem{lem}[thm]{Lemma}
 
 \numberwithin{equation}{section}

\begin{document}

\title[Perturbation method for Einstein-Dirac equations]
 {Perturbation method for particlelike solutions of Einstein-Dirac equations}

\author{Simona Rota Nodari}

\address{Ceremade (UMR CNRS 7534)
Universit\'e Paris-Dauphine\newline
Place Mar\'echal Lattre de Tassigny\newline
75775 Paris Cedex 16
France}

\email{rotanodari@ceremade.dauphine.fr}


\date{October 2, 2009}

\begin{abstract}
The aim of this work is to prove by a perturbation method the existence of solutions of the coupled Einstein-Dirac equations for a static, spherically symmetric system of two fermions in a singlet spinor state. We relate the 
solutions of our equations to those of the nonlinear Choquard equation and we show that the nondegenerate solution of Choquard’s equation generates solutions for Einstein-Dirac equations.
\end{abstract}

\maketitle

\section{Introduction}
In this paper, we study the coupled Einstein-Dirac equations for a static, spherically symmetric system of two fermions in a singlet spinor state. 
Using numerical methods, F. Finster, J. Smoller and ST. Yau found,  in \cite{Finsmoyau}, particlelike solutions; our goal is to give a rigorous proof of  their existence by a perturbation method
\footnote{After completing this work, we learned from professor Joel Smoller that Erik J. Bird had proved the existence of small solutions of the Einstein-Dirac equations in his doctoral thesis in 2005 \cite{Bird}. His method is quite different from ours: he uses Schauder's fixed point theorem.}.

The Einstein-Dirac equations take the form
\begin{eqnarray}\label{eq:einsteindirac1}
&&(D-m)\psi=0\\
\label{eq:einsteindirac2}
&&R^i_j-\frac{1}{2}R\delta^i_j=-8\pi T^i_j
\end{eqnarray}
where $D$ denotes the Dirac operator, $\psi$ is the wave function of a fermion of mass $m$, $R^i_j$ is the Ricci curvature tensor, $R$ indicates the scalar curvature and, finally, $T^i_j$ is the energy-momentum tensor of the Dirac particle.\\
In \cite{Finsmoyau}, Finster, Smoller and Yau work with the Dirac operator into a static, spherically symmetric space-time where the metric, in polar coordinates $(t,r,\vartheta,\varphi)$, is given by 
\begin{equation}\label{def:metric1}
g_{ij}=diag\left(\frac{1}{T^2},-\frac{1}{A},-r^2,-r^2\sin^2\vartheta\right)
\end{equation}
\begin{equation}\label{def:metric2}
g^{ij}=diag\left(T^2,-A,-\frac{1}{r^2},-\frac{1}{r^2\sin^2\vartheta}\right)
\end{equation}
with $A=A(r)$, $T=T(r)$ positive functions; so, the Dirac operator can be written as 
\begin{equation}\label{eq:Dirac}
D=i\gamma^t\partial_t+\gamma^r\left(i\partial_r+\frac{i}{r}\left(1-A^{-1/2}\right)-\frac{i}{2}\frac{T'}{T}\right)+i\gamma^{\vartheta}\partial_{\vartheta}+i\gamma^{\varphi}\partial_{\varphi}
\end{equation}
with
\begin{eqnarray}\label{def:Diracmatrice}
\gamma^t&=&T\bar{\gamma}^0\\
\gamma^r&=&\sqrt{A}\left(\bar{\gamma}^1\cos\vartheta+\bar{\gamma}^2\sin\vartheta\cos\varphi+\bar{\gamma}^3\sin\vartheta\sin\varphi\right)\\
\gamma^{\vartheta}&=&\frac{1}{r}\left(-\bar{\gamma}^1\sin\vartheta+\bar{\gamma}^2\cos\vartheta\cos\varphi+\bar{\gamma}^3\cos\vartheta\sin\varphi\right)\\
\gamma^{\varphi}&=&\frac{1}{r\sin\vartheta}\left(-\bar{\gamma}^2\sin\varphi+\bar{\gamma}^3\cos\varphi\right)
\end{eqnarray}
where $\bar{\gamma}^i$ are the Dirac matrices in Minkowski space (see \cite{Finsmoyau}).

Moreover, Finster, Smoller and Yau are looking for solutions taking the form 
\begin{equation}\label{def:wave}
\psi=e^{-i\omega t}r^{-1}T^{1/2}\left(\begin{array}{c}\Phi_1\left(\begin{array}{c}1\\0\end{array}\right)\\i\Phi_2\sigma^r\left(\begin{array}{c}1\\0\end{array}\right)\end{array}\right),
\end{equation}
where $\sigma^r=\left(\bar{\sigma}^1\cos\vartheta+\bar{\sigma}^2\sin\vartheta\cos\varphi+\bar{\sigma}^3\sin\vartheta\sin\varphi\right)$ is a linear combination of the Pauli matrices $\bar{\sigma}^i$ and $\Phi_1(r)$, $\Phi_2(r)$ are radial real functions.

We remind also that the energy-momentum tensor is obtained as the variation of the classical Dirac action 
$$
S=\int{\bar{\psi}(D-m)\psi\sqrt{|g|}\,d^4x}
$$
and takes the form
\begin{eqnarray*}
T^i_j&=&\frac{1}{r^{2}}diag\left(2\omega T^2|\Phi|^2,-2\omega T^2|\Phi|^2+4T\frac{1}{r}\Phi_1\Phi_2+2mT\left(\Phi_1^2-\Phi_2^2\right),\right.\\
&&\left.-2T\frac{1}{r}\Phi_1\Phi_2,-2T\frac{1}{r}\Phi_1\Phi_2\right)
\end{eqnarray*}
(see \cite{Finsmoyau} for more details).

In this case, the coupled Einstein-Dirac equations can be written as 
\begin{eqnarray}\label{eq:einsteindiraceq1}
\sqrt{A}\Phi_1'&=&\frac{1}{r}\Phi_1-(\omega T+m)\Phi_2\\
\label{eq:einsteindiraceq2}
\sqrt{A}\Phi_2'&=&(\omega T-m)\Phi_1-\frac{1}{r}\Phi_2\\
\label{eq:einsteindiraceq3}
rA'&=&1-A-16\pi\omega T^2\left(\Phi_1^2+\Phi_2^2\right)\\
\label{eq:einsteindiraceq4}
2rA\frac{T'}{T}&=&A-1-16\pi\omega T^2\left(\Phi_1^2+\Phi_2^2\right)+32\pi\frac{1}{r} T\Phi_1\Phi_2+\nonumber\\
&&+16\pi m T\left(\Phi_1^2-\Phi_2^2\right)
\end{eqnarray}
with the normalization condition
\begin{equation}\label{eq:normalcond}
\int_0^{\infty}{|\Phi|^2\frac{T}{\sqrt A}\,dr}=\frac{1}{4\pi}.
\end{equation}
In order that the metric be asymptotically Minkowskian, Finster, Smoller and Yau assume that
$$
\lim_{r\rightarrow \infty}T(r)=1.
$$
Finally, they also require that the solutions have finite (ADM) mass; namely
$$
\lim_{r\rightarrow \infty}\frac{r}{2}(1-A(r))<\infty.
$$

In this paper, we will prove the existence of solutions of (\ref{eq:einsteindirac1}-\ref{eq:einsteindirac2}) in the form (\ref{def:wave}) by a perturbation method.\\
In particular, we follow the idea described by Ounaies in \cite{ounaies} (see also \cite{Guan} for a rigorous existence proof of nonlinear Dirac solitons based on Ounaies' approach). Ounaies, by a perturbation parameter, relates the solutions of a nonlinear Dirac equation to those of nonlinear Schrödinger equation. Imitating the idea of Ounaies, we relate the solutions of ours equations to those of nonlinear Choquard's equation (see \cite{liebcho}, \cite{Lionscho} for more details on Choquard's equation) and we obtain the following result.
\begin{thm}\label{th:solution} Given $0<\omega< m$ such that $m-\omega$ is sufficiently small, there exists a non trivial solution of (\ref{eq:einsteindiraceq1}-\ref{eq:einsteindiraceq4}).
\end{thm}

\noindent In Section \ref{section:EinsteinDirac}, we solve the Einstein-Dirac equations by means of the
perturbation method suggested by Ounaies; in particular in the first subsection we describe a useful rescaling and some properties of the operators involved, whereas in the second subsection we prove the existence of solutions generated by the solution of the Choquard equation. 

\section{Perturbation method for Einstein-Dirac equations}\label{section:EinsteinDirac}
First of all, we observe that writing $T(r)=1+t(r)$ and using equation (\ref{eq:einsteindiraceq3}), the coupled Einstein-Dirac equations become
\begin{eqnarray}\label{eq:einsteindiraceq2.1}
\sqrt{A}\Phi_1'&=&\frac{1}{r}\Phi_1-(\omega +m)\Phi_2- \omega t \Phi_2\\
\label{eq:einsteindiraceq2.2}
\sqrt{A}\Phi_2'&=&(\omega -m)\Phi_1+\omega t \Phi_1-\frac{1}{r}\Phi_2\\
\label{eq:einsteindiraceq2.3}
2rAt'&=&(A-1)(1+t)-16\pi\omega (1+t)^3\left(\Phi_1^2+\Phi_2^2\right)+\nonumber\\
&&+32\pi\frac{1}{r} (1+t)^2\Phi_1\Phi_2+16\pi m (1+t)^2\left(\Phi_1^2-\Phi_2^2\right)
\end{eqnarray}
where 
\begin{equation}\label{eq:defA}
A(r)=1-\frac{16\pi\omega}{r}\int_0^{r}{(1+t(s))^2\left(\Phi_1(s)^2+\Phi_2(s)^2\right)\,ds}:=1-\frac{16\pi\omega}{r}Q(r).
\end{equation}
Furthermore, because we want $A(r)>0$, we have that the following condition must be satisfied
\begin{equation}\label{eq:condQ}
0\leq\frac{Q(r)}{r}<\frac{1}{16\pi\omega}
\end{equation}
for all $r\in(0,\infty)$.

Now, to find a solution of the equations (\ref{eq:einsteindiraceq2.1}-\ref{eq:einsteindiraceq2.3}), we exploit the idea used by Ounaies in \cite{ounaies}. In particular, we proceed as follow: in a first step we use a rescaling argument to transform (\ref{eq:einsteindiraceq2.1}-\ref{eq:einsteindiraceq2.3}) in a perturbed system of the form
\begin{equation}\label{eq:systemperturbed}
\left\{
\begin{array}{l}
\sqrt{A\left(\varepsilon,\varphi,\chi,\tau\right)} \frac{d}{dr}\varphi-\frac{1}{r}\varphi+2m\chi+K_1\left(\varepsilon,\varphi,\chi,\tau \right)=0\\[5pt]
\sqrt{A\left(\varepsilon,\varphi,\chi,\tau\right)} \frac{d}{dr}\chi+\frac{1}{r}\chi+\varphi-m\varphi\tau+K_2\left(\varepsilon,\varphi,\chi,\tau \right)=0\\[5pt]
A\left(\varepsilon,\varphi,\chi,\tau\right)\frac{d}{dr}\tau+\frac{8\pi m}{r^2}\int_0^r\varphi^2\,ds+K_3\left(\varepsilon,\varphi,\chi,\tau \right)=0
\end{array}
\right.
\end{equation}
where $\varphi,\chi,\tau:(0,\infty)\rightarrow\mathbb{R}$.\\
Second, we relate the solutions of (\ref{eq:systemperturbed}) to those of the nonlinear system 
\begin{equation}\label{eq:systemperturbed2}
\left\{
\begin{array}{l}
-\frac{d^2}{dr^2}\varphi+2m\varphi-16\pi m^3\left(\int_{0}^{\infty}\frac{\varphi^2}{\max(r,s)}\,ds\right)\varphi=0\\[5pt]
\chi(r)=\frac{1}{2m}\left(\frac{1}{r}\varphi-\frac{d}{dr}\varphi\right)\\[5pt]
\tau(r)=8\pi m \int_{0}^{\infty}\frac{\varphi^2}{\max(r,s)}\,ds.
\end{array}
\right.
\end{equation}
We remark that $\varphi$ is a solution of (\ref{eq:systemperturbed2}) if and only if $u(x)=\frac{\varphi(|x|)}{|x|}$ solves the nonlinear Choquard equation 
\begin{equation}\label{eq:choquard}
\begin{array}{cc}
-\triangle u+2m u-4m^3\left(\int_{\mathbb{R}^3}\frac{\left|u(y)\right|^2}{|x-y|}\,dy\right)u=0&\mbox{in}\ H^1\left(\mathbb{R}^3\right).
\end{array}
\end{equation}
To prove this fact it's enough to remind that for a radial function $\rho$, $$\triangle \rho =\frac{1}{r^2}\frac{d}{dr}\left(r^2\frac{d}{dr}\rho\right)$$
and $$\left(|\cdot|\star\rho\right)(x)=4\pi\left(\int_{0}^{\infty}\frac{s^2\rho(s)}{\max(r,s)}\,ds\right)$$
with $r=|x|$.\\
We observe also that if we write  
 \begin{equation*}
\left(\begin{array}{c}u(x)\\ v(x)\end{array}\right)=r^{-1}\left(\begin{array}{c}\varphi(r)\left(\begin{array}{c}1\\0\end{array}\right)\\i\chi(r)\sigma^r\left(\begin{array}{c}1\\0\end{array}\right)\end{array}\right)
\end{equation*}
with $r=|x|$, $(\varphi,\chi)$ is a solution of (\ref{eq:systemperturbed2}) if and only if $(u(x),v(x))$ solve
\begin{equation}\label{eq:choquardR3}
\begin{array}{cc}
-\triangle u+2m u-4m^3\left(\int_{\mathbb{R}^3}\frac{\left|u(y)\right|^2}{|x-y|}\,dy\right)u=0&v=\frac{-i\bar{\sigma}\nabla u}{2m}
\end{array}
\end{equation}
in $\mathbb{R}^3$ where $\bar{\sigma}\nabla=\sum_{i=1}^{3}\bar{\sigma}^i\partial_i$.\\
It's well known that Choquard's equation (\ref{eq:choquard}) has a unique radial, positive solution $u_0$ with $\int|u_0|^2=N$ for some $N>0$ given. Furthermore, $u_0$ is infinitely differentiable and goes to zero at infinity; more precisely there exist some positive constants $C_{\delta,\eta}$ such that $\left|D^{\eta}\left(u_0\right)\right|\leq C_{\delta,\eta}\exp(-\delta|x|)$ for $x\in\mathbb{R}^3$. At last, $u_0\in H^1(\mathbb{R}^3)$ is a radial nondegenerate solution; by this we mean that the linearization of (\ref{eq:choquard}) around $u_0$ has a trivial nullspace in $L^2_{r}(\mathbb{R}^3)$. In particular, the linear operator $\mathcal{L}$ given by
\begin{equation*}
\mathcal{L}\xi=-\triangle \xi+2m \xi-4m^3\left(\int_{\mathbb{R}^3}\frac{\left|u_0(y)\right|^2}{|x-y|}\,dy\right)\xi-8m^3\left(\int_{\mathbb{R}^3}\frac{\xi(y)u_0(y)}{|x-y|}\,dy\right)u_0
\end{equation*}
satisfies $\ker\mathcal{ L}=\{0\}$ when $\mathcal{L}$ is restricted to $L^2_{r}(\mathbb{R}^3)$
(see \cite{liebcho}, \cite{Lionscho}, \cite{Lenzmann} for more details).

The main idea is that the solutions of (\ref{eq:systemperturbed}) are the zeros of a $\mathcal{C}^1$ operator $D:\mathbb{R}\times X_{\varphi}\times X_{\chi}\times X_{\tau}\rightarrow Y_{\varphi}\times Y_{\chi}\times Y_{\tau}$. If we denote by $D_{\varphi,\chi,\tau}(\varepsilon,\varphi,\chi,\tau)$ the derivative of $D(\varepsilon,\cdot,\cdot,\cdot)$, by $(\varphi_0,\chi_0,\tau_0)$ the ground state solution of (\ref{eq:systemperturbed2}) and we observe that $D_{\varphi,\chi,\tau}(\varepsilon,\varphi_0,\chi_0,\tau_0)$ is an isomorphism, the application of the implicit function theorem (see \cite{RenRog}) yields the following result, which is equivalent to theorem \ref{th:solution}. 
\begin{thm}\label{th:principalth} Let $(\varphi_0,\chi_0,\tau_0)$ be the ground state solution of (\ref{eq:systemperturbed2}), then there exists $\delta>0$ and a function $\eta\in\mathcal{C}((0,\delta),X_{\varphi}\times X_{\chi}\times X_{\tau})$ such that $\eta(0)=(\varphi_0,\chi_0,\tau_0)$ and $(\varepsilon,\eta(\varepsilon))$ is a solution of (\ref{eq:systemperturbed}), for $0\leq \varepsilon <\delta$.
\end{thm}

\subsection{Rescaling}
In this subsection we are going to introduce the new variable $(\varphi,\chi,\tau)$ such that $\Phi_1(r)=\alpha\varphi(\lambda r)$, $\Phi_2(r)=\beta\chi(\lambda r)$ and $t(r)=\gamma\tau(\lambda r)$, where $\Phi_1,\Phi_2, t$ satisfy  (\ref{eq:einsteindiraceq2.1}-\ref{eq:einsteindiraceq2.3}) and $\alpha,\beta,\gamma,\lambda>0$ are constants to be chosen later.

Using the explicit expressions of $A$, given in (\ref{eq:defA}), we have
\begin{eqnarray}\label{eq:Aalphabetagamma}
A(\Phi_1,\Phi_2,t)&=&1-\frac{16\pi\omega\alpha^2}{r}\int_0^{r}{(1+\gamma \tau)^2\left(\varphi^2+\left(\frac{\beta}{\alpha}\chi\right)^2\right)\,ds}\nonumber\\
&:=&A_{\alpha,\beta,\gamma}(\varphi,\chi,\tau)
\end{eqnarray}
It's now clear that if $\Phi_1,\Phi_2,t$ satisfy (\ref{eq:einsteindiraceq2.1}-\ref{eq:einsteindiraceq2.3}), then $\varphi,\chi,\tau$ satisfy the system
\begin{equation}\label{eq:systemalphabeta}
\left\{
\begin{array}{l}
\sqrt{A_{\alpha,\beta,\gamma}}\frac{\alpha\lambda}{\beta}\frac{d}{dr}\varphi-\frac{\alpha\lambda}{\beta}\frac{1}{r}\varphi+(m+\omega)\chi+\omega \gamma\tau\chi=0\\[5pt]
\sqrt{A_{\alpha,\beta,\gamma}}\frac{d}{dr}\chi+\frac{1}{r}\chi+\frac{\alpha}{\beta\lambda}(m-\omega)\varphi-\frac{\alpha\gamma}{\beta\lambda}\omega\tau\varphi=0\\[5pt]
2A_{\alpha,\beta,\gamma}r\frac{d}{dr}\tau+K_{\alpha,\beta,\gamma}=0
\end{array}
\right.
\end{equation}
with
\begin{eqnarray*}
&&K_{\alpha,\beta,\gamma}(\varphi,\chi,\tau)=\frac{16\pi\omega}{r}\frac{\alpha^2}{\gamma}\left(\int_0^{r}{(1+\gamma \tau)^2\left(\varphi^2+\left(\frac{\beta}{\alpha}\chi\right)^2\right)\,ds}\right)(1+\gamma \tau)+\\
&&+16\pi\omega\frac{\alpha^2}{\gamma} (1+\gamma\tau)^3\left(\varphi^2+\left(\frac{\beta}{\alpha}\chi\right)^2\right)-32\pi\frac{\lambda\alpha\beta}{\gamma}\frac{1}{r} (1+\gamma\tau)^2\varphi\chi+\\
&&-16\pi m \frac{\alpha^2}{\gamma} (1+\gamma\tau)^2\left(\varphi^2-\left(\frac{\beta}{\alpha}\chi\right)^2\right).
\end{eqnarray*}

By adding the conditions $\frac{\alpha}{\beta\lambda}(m-\omega)=1$, $\frac{\alpha\lambda}{\beta}=1$, $\frac{\alpha\gamma}{\beta\lambda}=1$, $\frac{\alpha^2}{\gamma}=1$ and $m-\omega\geq 0$, we obtain $\alpha=(m-\omega)^{1/2}$, $\lambda=(m-\omega)^{1/2}$, $\beta=m-\omega$ and $\gamma=m-\omega$.

Denoting $\varepsilon=m-\omega$, (\ref{eq:systemalphabeta}) is equivalent to 
\begin{equation}\label{eq:systemepsilon}
\left\{
\begin{array}{l}
\sqrt{A\left(\varepsilon,\varphi,\chi,\tau\right)} \frac{d}{dr}\varphi-\frac{1}{r}\varphi+2m\chi+K_1\left(\varepsilon,\varphi,\chi,\tau \right)=0\\[5pt]
\sqrt{A\left(\varepsilon,\varphi,\chi,\tau\right)} \frac{d}{dr}\chi+\frac{1}{r}\chi+\varphi-m\varphi\tau+K_2\left(\varepsilon,\varphi,\chi,\tau \right)=0\\[5pt]
A\left(\varepsilon,\varphi,\chi,\tau\right)\frac{d}{dr}\tau+\frac{8\pi m}{r^2}\int_0^r\varphi^2\,ds+K_3\left(\varepsilon,\varphi,\chi,\tau \right)=0
\end{array}
\right.
\end{equation}
where $A\left(\varepsilon,\varphi,\chi,\tau\right)$, $K_1\left(\varepsilon,\varphi,\chi,\tau \right)$, $K_2\left(\varepsilon,\varphi,\chi,\tau \right)$ and $K_3\left(\varepsilon,\varphi,\chi,\tau \right)$ are defined by
\begin{eqnarray}\label{eq:Aepsilon}
&&A\left(\varepsilon,\varphi,\chi,\tau\right)=1-\frac{16\pi(m-\varepsilon)\varepsilon}{r}\int_0^{r}{(1+\varepsilon \tau)^2\left(\varphi^2+\varepsilon\chi^2\right)\,ds};\\
\label{eq:K1}
&&K_1\left(\varepsilon,\varphi,\chi,\tau \right)=-\varepsilon\chi+\varepsilon(m-\varepsilon)\tau\chi;\\
\label{eq:K2}
&&K_2\left(\varepsilon,\varphi,\chi,\tau \right)=\varepsilon\tau\varphi;
\end{eqnarray}
and
\begin{eqnarray}\label{eq:K3}
K_3\left(\varepsilon,\varphi,\chi,\tau \right)&=&\frac{8\pi m\varepsilon}{r^2}\int_0^r\chi^2\,ds+\frac{16\pi m\varepsilon}{r^2}\int_0^r\tau\left(\varphi^2+\varepsilon\chi^2\right)\,ds+\nonumber\\
&&+\frac{8\pi m\varepsilon^2}{r^2}\int_0^r\tau^2\left(\varphi^2+\varepsilon\chi^2\right)\,ds+\nonumber\\
&&-\frac{8\pi\varepsilon}{r^2}\left(\int_0^r(1+\varepsilon\tau)^2\left(\varphi^2+\varepsilon\chi^2\right)\,ds\right)\tau+\nonumber\\
&&+\frac{8\pi( m-\varepsilon)\varepsilon}{r^2}\left(\int_0^r(1+\varepsilon\tau)^2\left(\varphi^2+\varepsilon\chi^2\right)\,ds\right)\tau+\nonumber\\
&&+16\pi m\varepsilon\frac{\chi^2}{r}+8\pi m\varepsilon(3\tau+3\varepsilon\tau^2+\varepsilon^2\tau^3)\frac{\left(\varphi^2+\varepsilon\chi^2\right)}{r}+\nonumber\\
&&-8\pi \varepsilon(1+\varepsilon\tau)^3\frac{\left(\varphi^2+\varepsilon\chi^2\right)}{r}-16\pi \varepsilon(1+\varepsilon\tau)^2\frac{\varphi\chi}{r^2}+\nonumber\\
&&-8\pi m\varepsilon(2\tau+\varepsilon\tau^2)\frac{\left(\varphi^2-\varepsilon\chi^2\right)}{r}.
\end{eqnarray}

For $\varepsilon=0$, (\ref{eq:systemepsilon}) becomes
\begin{equation}\label{eq:systemepsilonzero}
\left\{
\begin{array}{l}
\frac{d}{dr}\varphi-\frac{1}{r}\varphi+2m\chi=0\\[5pt]
\frac{d}{dr}\chi+\frac{1}{r}\chi+\varphi-m\varphi\tau=0\\[5pt]
\frac{d}{dr}\tau+\frac{8\pi m}{r^2}\int_0^r\varphi^2\,ds=0
\end{array}
\right.
\end{equation}
that is equivalent to
\begin{equation}\label{eq:systemepsilonzero2}
\left\{
\begin{array}{l}
-\frac{d^2}{dr^2}\varphi+2m\varphi-16\pi m^3\left(\int_{0}^{\infty}\frac{\varphi^2}{\max(r,s)}\,ds\right)\varphi=0\\[5pt]
\chi(r)=\frac{1}{2m}\left(\frac{1}{r}\varphi-\frac{d}{dr}\varphi\right)\\[5pt]
\tau(r)=8\pi m \int_{0}^{\infty}\frac{\varphi^2}{\max(r,s)}\,ds.
\end{array}
\right.
\end{equation}
Then, we denote by $(\varphi_0,\chi_0,\tau_0)$ a solution of (\ref{eq:systemepsilonzero2}); in particular 
\begin{eqnarray*}
\chi_0(r)&=&-\frac{r}{2m}\frac{d}{dr}\left(\frac{\varphi_0}{r}\right)\\
\tau_0(r)&=&8\pi m \int_{0}^{\infty}\frac{\varphi_0^2}{\max(r,s)}\,ds.
\end{eqnarray*}

Now, to obtain a solution of  (\ref{eq:systemepsilon}) from $(\varphi_0,\chi_0,\tau_0)$, we define the operators  $L_1:\mathbb{R}\times X_{\varphi}\times X_{\chi}\times X_{\tau}\rightarrow Y_\varphi$, $L_2:\mathbb{R}\times X_{\varphi}\times X_{\chi}\times X_{\tau}\rightarrow Y_\chi$, $L_3:\mathbb{R}\times X_{\varphi}\times X_{\chi}\times X_{\tau}\rightarrow Y_\tau$ and $D:\mathbb{R}\times X_{\varphi}\times X_{\chi}\times X_{\tau}\rightarrow Y_{\varphi}\times Y_{\chi}\times Y_{\tau}$ by 
\begin{eqnarray*}
L_1(\varepsilon,\varphi,\chi,\tau)&=&\sqrt{A\left(\varepsilon,\varphi,\chi,\tau\right)}\frac{1}{r} \frac{d}{dr}\varphi-\frac{\varphi}{r^2}+2m\frac{\chi}{r}+\frac{1}{r}K_1\left(\varepsilon,\varphi,\chi,\tau \right)\\
L_2(\varepsilon,\varphi,\chi,\tau)&=&\sqrt{A\left(\varepsilon,\varphi,\chi,\tau\right)}\frac{1}{r} \frac{d}{dr}\chi+\frac{\chi}{r^2}+\frac{\varphi}{r}-m\frac{\varphi}{r}\tau+\frac{1}{r}K_2\left(\varepsilon,\varphi,\chi,\tau \right)\\
L_3(\varepsilon,\varphi,\chi,\tau)&=&A\left(\varepsilon,\varphi,\chi,\tau\right)\frac{d}{dr}\tau+\frac{8\pi m}{r^2}\int_0^r\varphi^2\,ds+K_3\left(\varepsilon,\varphi,\chi,\tau \right)
\end{eqnarray*}
and
\begin{equation*}
D(\varepsilon,\varphi,\chi,\tau)=(L_1(\varepsilon,\varphi,\chi,\tau),L_2(\varepsilon,\varphi,\chi,\tau),L_3(\varepsilon,\varphi,\chi,\tau))
\end{equation*}
where 
\begin{eqnarray*}
X_\varphi&=&\left\{\varphi:(0,\infty)\rightarrow\mathbb{R}\left|\frac{\varphi(|x|)}{|x|}\left(\begin{array}{c}1\\0\end{array}\right)\in H^1\left(\mathbb{R}^3,\mathbb{R}^2\right)\right.\right\}\\
X_\chi&=&\left\{\chi:(0,\infty)\rightarrow\mathbb{R}\left|\frac{\chi(|x|)}{|x|}\sigma^r\left(\begin{array}{c}1\\0\end{array}\right)\in H^1\left(\mathbb{R}^3,\mathbb{C}^2\right)\right.\right\}\\
X_\tau&=&\left\{\tau:(0,\infty)\rightarrow\mathbb{R}\left|\lim_{r\rightarrow\infty}\tau(r)\rightarrow 0, \frac{d}{dr}\tau\in L^1((0,\infty),dr)\right.\right\}\\
Y_\varphi&=&Y_\chi=L^2\left(\mathbb{R}^3\right)\\
Y_\tau&=&L^1((0,\infty),dr).
\end{eqnarray*}
Furthermore we define the following norms:
\begin{eqnarray*}
\|\varphi\|_{X_\varphi}&=&\left\|\frac{\varphi(|x|)}{|x|}\right\|_{H^1\left(\mathbb{R}^3\right)},\\
\|\chi\|_{X_\chi}&=&\left\|\frac{\chi(|x|)}{|x|}\sigma^r\left(\begin{array}{c}1\\0\end{array}\right)\right\|_{H^1\left(\mathbb{R}^3\right)},\\
\|\tau\|_{X_\tau}&=&\left\|\frac{d}{dr}\tau\right\|_{L^1((0,\infty),dr)}.\\
\end{eqnarray*}

It's well known that
$$
\begin{array}{ll}
H^1\left(\mathbb{R}^3\right)\hookrightarrow L^q\left(\mathbb{R}^3\right)& 2\leq q \leq 6\\[5pt]
X_\tau \hookrightarrow L^\infty\left((0,\infty),dr\right).\\[5pt]
\end{array}
$$
Moreover, using Hardy's inequality
$$
\int_{\mathbb{R}^3}{\frac{|f|^2}{|x|^2}\,dx}\leq 4\int_{\mathbb{R}^3}{|\nabla f|^2\,dx},
$$
we get the following properties:
\begin{equation}\label{eq:condsfunct2}
\begin{array}{c}
\rho \in H^1\left((0,\infty),dr\right)\hookrightarrow L^\infty\left((0,\infty),dr\right) \\[5pt]
\frac{\rho}{r}\in L^2\left((0,\infty),dr\right).
\end{array}
\end{equation}
$ \forall \rho \in X_\varphi, \forall \rho\in X_\chi$.

Since the operator $A(\varepsilon, \varphi, \chi, \tau)$ must be strictly positive, we consider $B_\varphi$, $B_\chi$, $B_\tau$, defined as the balls of the spaces $X_\varphi,X_\chi,X_\tau$, and $\varepsilon_1,\varepsilon_2$, depending on $m$ and on the radius of $B_\varphi, B_\chi, B_\tau$, such that 
$$
1-\frac{16\pi(m-\varepsilon)\varepsilon}{r}\int_0^{r}{(1+\varepsilon \tau)^2\left(\varphi^2+\varepsilon\chi^2\right)\,ds}\geq\delta>0
$$ 
for all $(\varepsilon,\varphi,\chi,\tau)\in (-\varepsilon_1,\varepsilon_2)\times B_{\varphi}\times B_{\chi}\times B_{\tau}$. The existence of $\varepsilon_1,\varepsilon_2$ is assured by the fact that $\varphi,\chi,\tau$ are bounded; in particular, if $\varepsilon\geq0$,
\begin{eqnarray*}
&&1-\frac{16\pi(m-\varepsilon)\varepsilon}{r}\int_0^{r}{(1+\varepsilon \tau)^2\left(\varphi^2+\varepsilon\chi^2\right)\,ds}\geq\\
&&\geq1-20m\varepsilon\left\|\varphi\right\|^2_{X_\varphi}-8m\varepsilon^2\left(5\left\|\tau\right\|_{X_\tau}\left\|\varphi\right\|^2_{X_\varphi}+\left\|\chi\right\|^2_{X_\chi}\right)+\\
&&-4m\varepsilon^3\left\|\tau\right\|_{X_\tau}\left(5\left\|\tau\right\|_{X_\tau}\left\|\varphi\right\|^2_{X_\varphi}+4\left\|\chi\right\|^2_{X_\chi}\right)-8m\varepsilon^4\left\|\tau\right\|_{X_\tau}^2\left\|\chi\right\|^2_{X_\chi},\\
\end{eqnarray*}
then there exists $\varepsilon_2>0$ such that $A(\varepsilon, \varphi, \chi, \tau)>0$ for all $\varepsilon\in [0,\varepsilon_2)$. In the same way, if $\varepsilon <0$,
\begin{eqnarray*}
&&1-\frac{16\pi(m-\varepsilon)\varepsilon}{r}\int_0^{r}{(1+\varepsilon \tau)^2\left(\varphi^2+\varepsilon\chi^2\right)\,ds}\geq\\
&&\geq1-8m\varepsilon^2\left\|\chi\right\|^2_{X_\chi}-8m|\varepsilon|^3\left\|\chi\right\|^2_{X_\chi}\left(1+2\left\|\tau\right\|_{X_\tau}\right)+\\
&&-8m\varepsilon^4\left\|\chi\right\|^2_{X_\chi}\left\|\tau\right\|_{X_\tau}\left(2+1\left\|\tau\right\|_{X_\tau}\right)-8m|\varepsilon|^5\left\|\tau\right\|_{X_\tau}^2\left\|\chi\right\|^2_{X_\chi},\\
\end{eqnarray*}
then there exists $\varepsilon_1>0$ such that $A(\varepsilon, \varphi, \chi, \tau)>0$ for all $\varepsilon\in (-\varepsilon_1,0)$.
  
\begin{lem}\label{lemma:regularity} The operators $L_1,L_2\in\mathcal{C}^1\left( (-\varepsilon_1,\varepsilon_2)\times B_{\varphi}\times B_{\chi}\times B_{\tau},Y_\varphi\right)$ and $L_3\in\mathcal{C}^1\left( (-\varepsilon_1,\varepsilon_2)\times B_{\varphi}\times B_{\chi}\times B_{\tau},Y_\tau\right)$.
\end{lem}
Before starting the proof of the lemma we observe that for a radial function $\rho$ such that $\frac{\rho}{r}\in H^1_{r}\left(\mathbb{R}^3\right)$ we have 
\begin{equation}\label{eq:radialfunction}
\left|\rho(r)\right|\leq r^{1/2}\left\|\frac{d}{dr}\left(\frac{\rho(r)}{r}\right)\right\|_{L^2_{rad}}.
\end{equation}
We remind that $H^1_{r}\left(\mathbb{R}^3\right)=\left\{u\in H^1\left(\mathbb{R}^3\right)\left|\ u\ \mbox{is radial}\right.\right\}$.

\begin{proof}
We begin with $L_3$; first, we have to prove that it is well defined in $Y_\tau=L^1((0,\infty),dr)$. We remark that 
\begin{eqnarray*}
\left|L_3\left(\varepsilon,\varphi,\chi,\tau \right)\right|&\leq& C_1\left|\frac{d}{dr}\tau\right|+\frac{C_2}{r^2}\int_0^r\left|\varphi^2+\varepsilon\chi^2\right|\,ds+\frac{C_3}{r}\left|\varphi^2+\varepsilon\chi^2\right|+\\
&&+\frac{C_4}{r^2}\left|\varphi\chi\right|+\frac{C_5}{r}\left|\varphi^2-\varepsilon\chi^2\right|
\end{eqnarray*}
where $C_1,C_2,C_3,C_4,C_5$ are positive constants and, by definition, we have that $\frac{d}{dr}\tau \in L^1((0,\infty),dr)$.\\
Next, we have
\begin{equation*}
\int_0^{\infty}{\frac{1}{r^2}\int_0^r\left|\varphi^2+\varepsilon\chi^2\right|\,ds\,dr} =\int_0^{\infty}{\frac{\left|\varphi^2+\varepsilon\chi^2\right|}{s}\,ds}<+\infty,
\end{equation*}
using Hölder's inequality, then $\frac{1}{r^2}\int_0^r\left|\varphi^2+\varepsilon\chi^2\right|\,ds\in Y_\tau$. In the same way,
we can conclude that $\frac{1}{r}\left(\varphi^2+\varepsilon\chi^2\right), \frac{1}{r}\left(\varphi^2-\varepsilon\chi^2\right)\in Y_\tau$.\\
Finally, 
\begin{equation*}
\int_0^{\infty}{\frac{|\varphi|}{r}\frac{|\chi|}{r}\,dr}\leq C\left\|\frac{\varphi}{r}\right\|_{L^2((0,\infty))}\left\|\frac{\chi}{r}\right\|_{L^2((0,\infty))}<+\infty
\end{equation*}
thanks to (\ref{eq:condsfunct2}), then $\frac{1}{r^2}\varphi\chi\in Y_\tau$.\\
Now, we have to prove that $L_3\left(\varepsilon,\varphi,\chi,\tau \right)$ is $\mathcal{C}^1$; by classical arguments, it's enough to show that for $(h_1,h_2,h_3)\in B_{\varphi}\times B_{\chi}\times B_{\tau}$
\begin{eqnarray*}
\frac{\partial}{\partial\varphi}\left(L_3\left(\varepsilon,\varphi,\chi,\tau \right)\right)h_1\in Y_\tau,\\
\frac{\partial}{\partial\chi}\left(L_3\left(\varepsilon,\varphi,\chi,\tau \right)\right)h_2\in Y_\tau,\\
\frac{\partial}{\partial\tau}\left(L_3\left(\varepsilon,\varphi,\chi,\tau \right)\right)h_3\in Y_\tau.
\end{eqnarray*} 
We begin with $\frac{\partial}{\partial\varphi}\left(L_3\left(\varepsilon,\varphi,\chi,\tau \right)\right)$,
\begin{eqnarray*}
&&\frac{\partial}{\partial\varphi}\left(L_3\left(\varepsilon,\varphi,\chi,\tau \right)\right)h_1=\left(\frac{\partial}{\partial\varphi}\left(A\left(\varepsilon,\varphi,\chi,\tau \right)\right)h_1\right)\frac{d}{dr}\tau+\\
&&+\frac{16\pi( m-\varepsilon)}{r^2}\left(\int_0^r(1+\varepsilon\tau)^2\varphi h_1\,ds\right)(1+\varepsilon\tau)+\\
&&+16\pi(m- \varepsilon)(1+\varepsilon\tau)^3\frac{\varphi h_1}{r}-16\pi \varepsilon(1+\varepsilon\tau)^2\frac{h_1\chi}{r^2}+\\
&&-16\pi m(1+\varepsilon\tau)^2\frac{\varphi h_1}{r};
\end{eqnarray*}
for $\frac{\partial}{\partial\chi}\left(L_3\left(\varepsilon,\varphi,\chi,\tau \right)\right)$,
\begin{eqnarray*}
&&\frac{\partial}{\partial\chi}\left(L_3\left(\varepsilon,\varphi,\chi,\tau \right)\right)h_2=\left(\frac{\partial}{\partial\chi}\left(A\left(\varepsilon,\varphi,\chi,\tau \right)\right)h_2\right)\frac{d}{dr}\tau+\\
&&+\frac{16\pi( m-\varepsilon)\varepsilon}{r^2}\left(\int_0^r(1+\varepsilon\tau)^2\chi h_2\,ds\right)(1+\varepsilon\tau)+\\
&&+16\pi(m- \varepsilon)\varepsilon(1+\varepsilon\tau)^3\frac{\chi h_2}{r}-16\pi \varepsilon(1+\varepsilon\tau)^2\frac{\varphi h_2}{r^2}+\\
&&-16\pi m\varepsilon(1+\varepsilon\tau)^2\frac{\chi h_2}{r}
\end{eqnarray*}
and, finally,
\begin{eqnarray*}
&&\frac{\partial}{\partial\tau}\left(L_3\left(\varepsilon,\varphi,\chi,\tau \right)\right)h_3=\left(\frac{\partial}{\partial\tau}\left(A\left(\varepsilon,\varphi,\chi,\tau \right)\right)h_3\right)\frac{d}{dr}\tau+\\
&&+A\left(\varepsilon,\varphi,\chi,\tau \right)\frac{d}{dr}h_3+\\
&&+\frac{16\pi( m-\varepsilon)\varepsilon}{r^2}\left(\int_0^r(1+\varepsilon\tau)h_3\left(\varphi^2+\varepsilon\chi^2\right)\,ds\right)(1+\varepsilon\tau)+\\
&&+\frac{8\pi( m-\varepsilon)\varepsilon}{r^2}\left(\int_0^r(1+\varepsilon\tau)^2\left(\varphi^2+\varepsilon\chi^2\right)\,ds\right)h_3+\\
&&+24\pi(m- \varepsilon)\varepsilon(1+\varepsilon\tau)^2\frac{\left(\varphi^2+\varepsilon\chi^2\right)}{r}h_3-32\pi \varepsilon^2(1+\varepsilon\tau)\frac{\varphi \chi}{r^2}h_3+\\
&&-16\pi m\varepsilon(1+\varepsilon\tau)\frac{\left(\varphi^2-\varepsilon\chi^2\right)}{r}h_3.
\end{eqnarray*}
First of all, we remark that if $\varphi,h_1\in B_\varphi$, $\chi,h_2\in B_\chi$ and $\tau,h_3\in B_\tau$, then $\frac{\partial}{\partial\varphi}\left(A\left(\varepsilon,\varphi,\chi,\tau \right)\right)h_1$, $\frac{\partial}{\partial\chi}\left(A\left(\varepsilon,\varphi,\chi,\tau \right)\right)h_2$ and $\frac{\partial}{\partial\tau}\left(A\left(\varepsilon,\varphi,\chi,\tau \right)\right)h_3$ are bounded. So, we have that  
\begin{eqnarray*}
\left|\frac{\partial L_3}{\partial\varphi}h_1\right|&\leq&C_1\left|\frac{d}{dr}\tau\right|+\frac{C_2}{r^2}\left(\int_0^r\left|\varphi h_1\right|\,ds\right)+C_3\frac{\left|\varphi h_1\right|}{r}+C_4\frac{|h_1\chi|}{r^2}\\
\left|\frac{\partial L_3}{\partial\chi}h_2\right|&\leq&C_5\left|\frac{d}{dr}\tau\right|+\frac{C_6}{r^2}\left(\int_0^r\left|\chi h_2\right|\,ds\right)+C_7\frac{\left|\chi h_2\right|}{r}+C_8\frac{|\varphi h_2|}{r^2}\\
\left|\frac{\partial L_3}{\partial\tau}h_3\right|&\leq&C_9\left|\frac{d}{dr}\tau\right|+C_{10}\left|\frac{d}{dr}h_3\right|+\frac{C_{11}}{r^2}\int_0^r\left|\varphi^2+\varepsilon\chi^2\right|\,ds+C_{12}\frac{\left|\varphi^2+\varepsilon\chi^2\right|}{r}+\\
&&+C_{13}\frac{|\varphi \chi|}{r^2}+C_{14}\frac{\left|\varphi^2-\varepsilon\chi^2\right|}{r}
\end{eqnarray*}
with $C_i$ positive constants.
With exactly the same arguments used above, we conclude that
\begin{eqnarray*}
\int_0^{\infty}{\left|\frac{\partial}{\partial\varphi}\left(L_3\left(\varepsilon,\varphi,\chi,\tau \right)\right)h_1 \right|\,dr}<+\infty\\
\int_0^{\infty}{\left|\frac{\partial}{\partial\chi}\left(L_3\left(\varepsilon,\varphi,\chi,\tau \right)\right)h_2 \right|\,dr}<+\infty\\
\int_0^{\infty}{\left|\frac{\partial}{\partial\tau}\left(L_3\left(\varepsilon,\varphi,\chi,\tau \right)\right)h_3 \right|\,dr}<+\infty\\
\end{eqnarray*}
if $(\varepsilon,\varphi,\chi,\tau)\in (-\varepsilon_1,\varepsilon_2)\times B_{\varphi}\times B_{\chi}\times B_{\tau},$ and $(h_1,h_2,h_3)\in B_{\varphi}\times B_{\chi}\times B_{\tau}$.\\
Furthermore $\frac{\partial L_3}{\partial\varphi},\frac{\partial L_3}{\partial\chi}$ and $\frac{\partial L_3}{\partial\tau}$ are continuous; thus the proof for $L_3$.

We consider now $L_1$; first, we have to prove that it is well defined in $Y_\varphi$. We observe that 
\begin{equation*}
\left|L_1\left(\varepsilon,\varphi,\chi,\tau \right)\right|\leq C_1\left|\frac{1}{r}\frac{d}{dr}\varphi\right|+\left|\frac{\varphi}{r^2} \right|+C_2\left|\frac{\chi}{r}\right|
\end{equation*}
with $C_1,C_2$ positive constants then, $L_1\left(\varepsilon,\varphi,\chi,\tau \right)\in L^2\left(\mathbb{R}^3\right)$, thanks to conditions (\ref{eq:condsfunct2}).\\
Now, we have to prove that $L_1\left(\varepsilon,\varphi,\chi,\tau \right)$ is $\mathcal{C}^1$; by classical arguments, it's enough to show that for $(h_1,h_2,h_3)\in B_{\varphi}\times B_{\chi}\times B_{\tau}$
\begin{eqnarray*}
\frac{\partial}{\partial\varphi}\left(L_1\left(\varepsilon,\varphi,\chi,\tau \right)\right)h_1\in Y_\varphi,\\
\frac{\partial}{\partial\chi}\left(L_1\left(\varepsilon,\varphi,\chi,\tau \right)\right)h_2\in Y_\varphi,\\
\frac{\partial}{\partial\tau}\left(L_1\left(\varepsilon,\varphi,\chi,\tau \right)\right)h_3\in Y_\varphi.
\end{eqnarray*} 
By a straightforward computation, we find out
\begin{eqnarray*}
&&\frac{\partial L_1}{\partial\varphi}h_1=\frac{1}{2}A^{-1/2}\left(\frac{\partial A}{\partial\varphi}h_1\right)\frac{1}{r}\frac{d}{dr}\varphi+A^{1/2}\frac{1}{r}\frac{d}{dr}h_1-\frac{h_1}{r^2},\\
&&\frac{\partial L_1}{\partial\chi}h_2=\frac{1}{2}A^{-1/2}\left(\frac{\partial A}{\partial\chi}h_2\right)\frac{1}{r}\frac{d}{dr}\varphi+(2m-\varepsilon)\frac{h_2}{r}+\varepsilon(m-\varepsilon)\tau \frac {h_2}{r},\\
&&\frac{\partial L_1}{\partial\tau}h_3=\frac{1}{2}A^{-1/2}\left(\frac{\partial A}{\partial\tau}h_3\right)\frac{1}{r}\frac{d}{dr}\varphi+\varepsilon(m-\varepsilon)h_3\frac{\chi}{r};
\end{eqnarray*}
and, using the positivity of $A$,
\begin{eqnarray*}
\left|\frac{\partial L_1}{\partial\varphi}h_1\right|&\leq&C_1\left|\frac{1}{r}\frac{d}{dr}\varphi\right|+C_2\left|\frac{1}{r}\frac{d}{dr}h_1\right|+\left|\frac{h_1}{r^2}\right|\\
\left|\frac{\partial L_1}{\partial\chi}h_2\right|&\leq&C_3\left|\frac{1}{r}\frac{d}{dr}\varphi\right|+C_4\left|\frac{h_2}{r}\right|\\
\left|\frac{\partial L_3}{\partial\tau}h_3\right|&\leq&C_5\left|\frac{1}{r}\frac{d}{dr}\varphi\right|+C_6\left|\frac{\chi}{r}\right|.
\end{eqnarray*}
with $C_i$ positive constants.
Then, we can conclude that
\begin{eqnarray*}
\int_{\mathbb{R}^3}{\left|\frac{\partial}{\partial\varphi}\left(L_3\left(\varepsilon,\varphi,\chi,\tau \right)\right)h_1 \right|^2\,dx}<+\infty\\
\int_{\mathbb{R}^3}{\left|\frac{\partial}{\partial\chi}\left(L_3\left(\varepsilon,\varphi,\chi,\tau \right)\right)h_2 \right|^2\,dx}<+\infty\\
\int_{\mathbb{R}^3}{\left|\frac{\partial}{\partial\tau}\left(L_3\left(\varepsilon,\varphi,\chi,\tau \right)\right)h_3 \right|^2\,dx}<+\infty\\
\end{eqnarray*}
if $(\varepsilon,\varphi,\chi,\tau)\in (-\varepsilon_1,\varepsilon_2)\times B_{\varphi}\times B_{\chi}\times B_{\tau},$ and $(h_1,h_2,h_3)\in B_{\varphi}\times B_{\chi}\times B_{\tau}$.\\
Furthermore $\frac{\partial L_1}{\partial\varphi},\frac{\partial L_1}{\partial\chi}$ and $\frac{\partial L_1}{\partial\tau}$ are continuous; thus the proof for $L_1$ and with the same arguments for $L_2$.\\
\end{proof}

\subsection{Branches generated by solutions of Choquard equation} 
In this subsection, we show that a solution $\phi_0=(\varphi_0,\chi_0,\tau_0)$ of (\ref{eq:systemperturbed2}) can generate a local branch of solutions of (\ref{eq:systemperturbed}). 

First, we linearize the operator $D$ on $(\varphi,\chi,\tau)$ around $(0,\phi_0)$
$$
D_{\varphi,\chi,\tau}(0,\phi_0)(h,k,l)=\left(
\begin{array}{c}
\frac{1}{r}\frac{d}{dr}h-\frac{h}{r^2}+2m\frac{k}{r}\\[5pt]
\frac{1}{r}\frac{d}{dr}k+\frac{k}{r^2}+\frac{h}{r}-m \frac{h}{r}\tau_0-m\frac{\varphi_0}{r} l\\[5pt]
\frac{d}{dr}l+\frac{16\pi m}{r^2}\int_0^r\varphi_0h\,ds
\end{array}
\right).
$$
Now, if we prove that $D_{\varphi,\chi,\tau}(0,\phi_0)$ is an isomorphism, the implicit function theorem can be applied and we can find solutions of (\ref{eq:systemperturbed}) near the ground state $
\phi_0$.

\begin{lem}\label{lemma:isoV} We define the operator $V:X_\varphi\times X_\chi\rightarrow Y_\varphi\times Y_\chi$, by 
$$
V(\varphi,\chi)=\left(\begin{array}{c}\frac{1}{r}\frac{d}{dr}\varphi-\frac{1}{r^2} \varphi +2m\frac{1}{r} \chi\\[5pt]
\frac{1}{r}\frac{d}{dr}\chi+\frac{1}{r^2}\chi+\frac{1}{r}\varphi \\
\end{array}\right),
$$
then $V$ is an isomorphism of $X_\varphi\times X_\chi$ onto $Y_\varphi\times Y_\chi$.
\end{lem}
This lemma is obvious if we remind that $L^2(\mathbb{R}^3,\mathbb{C}^4)$ can be written as the direct sum of partial wave subspaces and that the Dirac operator leaves invariant all these subspaces (see \cite{thaller}). So, thanks to lemma $2.3$ of \cite{ounaies}, we know that
$\overline V : H^1\left(\mathbb{R}^3,\mathbb{C}^2\right)\times H^1\left(\mathbb{R}^3,\mathbb{C}^2\right)\rightarrow L^2\left(\mathbb{R}^3,\mathbb{C}^2\right)\times L^2\left(\mathbb{R}^3,\mathbb{C}^2\right)$ defined by
$$
\overline V(u,v)=\left(\begin{array}{c}i\bar\sigma\nabla u+2m v\\ -i\bar\sigma\nabla v+ u \end{array}\right)
$$
is an isomorphism of $H^1\left(\mathbb{R}^3,\mathbb{C}^2\right)\times H^1\left(\mathbb{R}^3,\mathbb{C}^2\right)$ onto $L^2\left(\mathbb{R}^3,\mathbb{C}^2\right)\times L^2\left(\mathbb{R}^3,\mathbb{C}^2\right)$ and then 
$\overline{V}$ is an isomorphism of each partial wave subspace. In particular, $V$ coincide with $\overline V$ on the partial wave subspace $X_\varphi\times  X_\chi$.

\begin{lem}\label{lemma:iso} We define the operator $W:X_\varphi\times 
X_\chi\times X_\tau\rightarrow Y_\varphi\times Y_\chi\times Y_\tau$, by 
$$
W(h,k,l)=\left(
\begin{array}{c}
\frac{1}{r}\frac{d}{dr}h-\frac{h}{r^2}+2m\frac{k}{r}\\[5pt]
\frac{1}{r}\frac{d}{dr}k+\frac{k}{r^2}+\frac{h}{r}-m\frac{\varphi_0}{r} l\\[5pt]
\frac{d}{dr}l
\end{array}
\right),
$$
then $W$ is an isomorphism of $X_\varphi\times 
X_\chi\times X_\tau$ onto $Y_\varphi\times Y_\chi\times Y_\tau$.
\end{lem}
\begin{proof}
First we prove that $W$ is one to one. We observe that $W(h,k,l)=0$ if and only if $(h,k,l)$ satisfies
$$
\left\{
\begin{array}{l}
\frac{1}{r}\frac{d}{dr}h-\frac{h}{r^2}+2m\frac{k}{r}=0\\[5pt]
\frac{1}{r}\frac{d}{dr}k+\frac{k}{r^2}+\frac{h}{r}-m\frac{\varphi_0}{r} l=0\\[5pt]
\frac{d}{dr}l=0
\end{array}
\right.
$$
in $Y_\varphi\times Y_\chi \times Y_\tau$. In particular, we must have $l\equiv0$ and $(h,k)$ solution of
$$
\left\{
\begin{array}{l}
\frac{1}{r}\frac{d}{dr}h-\frac{h}{r^2}+2m\frac{k}{r}=0\\[5pt]
\frac{1}{r}\frac{d}{dr}k+\frac{k}{r^2}+\frac{h}{r}=0
\end{array}
\right.
$$
that is equivalent to $V(h,k)=0$. So, thank to lemma \ref{lemma:isoV}, $h\equiv k\equiv 0$ and $W$ is one to one in $Y_\varphi\times Y_\chi \times Y_\tau$.

Secondly, we have to prove that for $f=(f_1,f_2,f_3)\in Y_\varphi\times Y_\chi \times Y_\tau$, there exists $(h,k,l)\in X_\varphi\times 
X_\chi\times X_\tau$ such that $W(h,k,l)=f$. This means that the system 
$$
\left\{
\begin{array}{l}
\frac{1}{r}\frac{d}{dr}h-\frac{h}{r^2}+2m\frac{k}{r}=f_1\\[5pt]
\frac{1}{r}\frac{d}{dr}k+\frac{k}{r^2}+\frac{h}{r}-m\frac{\varphi_0}{r} l=f_2\\[5pt]
\frac{d}{dr}l=f_3
\end{array}
\right.
$$
has a solution in $X_\varphi\times X_\chi\times X_\tau$ for all $(f_1,f_2,f_3)\in Y_\varphi\times Y_\chi \times Y_\tau$.
We observe that $\forall f_3\in L^1((0,\infty),dr)$ there exist $l^*(r)=-\int_r^\infty f_3\,ds $ such that $\frac{d}{dr}l^*=f_3$; furthermore $l^*\in X_\tau$. So,
we have to show that 
\begin{equation}\label{eq:prooflemmaiso}
\left\{
\begin{array}{l}
\frac{1}{r}\frac{d}{dr}h-\frac{h}{r^2}+2m\frac{k}{r}=f_1\\[5pt]
\frac{1}{r}\frac{d}{dr}k+\frac{k}{r^2}+\frac{h}{r}=f_2+m\frac{\varphi_0}{r} l^*
\end{array}
\right.
\end{equation}
has a solution in $X_\varphi\times X_\chi$ for all $(f_1,f_2)\in Y_\varphi\times Y_\chi$.\\ Now, we remark that $\frac{\varphi_0}{r} l^*\in L^2\left(\mathbb{R}^3\right)$ and then, thanks to lemma \ref{lemma:isoV}, (\ref{eq:prooflemmaiso}) has a solution  in $X_\varphi\times X_\chi$ for all $(f_1,f_2)\in Y_\varphi\times Y_\chi$.\\
In conclusion $W$ is an  isomorphism of $X_\varphi\times X_\chi\times X_\tau$ onto $Y_\varphi\times Y_\chi\times Y_\tau$.\\
\end{proof}

Finally, we observe that $D_{\varphi,\chi,\tau}(0,\phi_0)(h,k,l)$ can be written as 
\begin{equation}
D_{\varphi,\chi,\tau}(0,\phi_0)(h,k,l)=W(h,k,l)+S(h)
\end{equation}
with
\begin{equation}\label{eq:defS}
S(h)=\left(
\begin{array}{c}
0\\[5pt]
-m \frac{h}{r}\tau_0\\[5pt]
\frac{16\pi m}{r^2}\int_0^r\varphi_0h\,ds\\[5pt]
\end{array}
\right).
\end{equation}

\begin{thm}\label{th:principalth1} Let $\phi_0$ be the ground state solution of (\ref{eq:systemperturbed2}), then there exists $\delta>0$ and a function $\eta\in\mathcal{C}((0,\delta),X_\varphi\times 
X_\chi\times X_\tau)$ such that $\eta(0)=\phi_0$ and $D(\varepsilon,\eta(\varepsilon))=0$ for $0\leq \varepsilon <\delta$.
\end{thm}
\begin{proof}
Since $D(0,\phi_0)=0$ and $D$ is continuously differentiable in a neighborhood of $(0,\phi_0)$, to apply the implicit function theorem we have to prove that $D_{\varphi,\chi,\tau}(0,\phi_0)$ is an isomorphism of $X_\varphi\times 
X_\chi\times X_\tau$ onto $Y_\varphi\times Y_\chi\times Y_\tau$.\\ We observe that $D_{\varphi,
\chi,\tau}(0,\phi_0)(h,k,l)=0$ if and only if $(h,k,l)$ satisfies
\begin{equation}\label{eq:systemderiv}
\left\{
\begin{array}{l}
\frac{d}{dr}h-\frac{h}{r}+2mk=0\\[5pt]
\frac{d}{dr}k+\frac{k}{r}+h-m h\tau_0-m\varphi_0 l=0\\[5pt]
\frac{d}{dr}l+\frac{16\pi m}{r^2}\int_0^r\varphi_0h\,ds=0
\end{array}
\right.
\end{equation}
that means
\begin{equation}\label{eq:systemderiv2}
\left\{
\begin{array}{l}
-\frac{d^2}{dr^2}h+2m h-16\pi m^3\left(\int_{0}^{\infty}\frac{\varphi_0^2}{\max(r,s)}\,ds\right)h-32\pi m^3\left(\int_{0}^{\infty}\frac{\varphi_0 h}{\max(r,s)}\,ds\right)\varphi_0=0\\\frac{d}{dr}h-
\frac{1}{r}h+2mk=0\\[5pt]
l=16\pi m\int_0^\infty\frac{\varphi_0 h}{\max(r,s)}\,ds
\end{array}
\right.
\end{equation}
Now, if we write $\xi(x)=\frac{h(|x|)}{|x|}$ and we remind that $\varphi_0(|x|)=|x|u_0(x)$ with $u_0$ solution of (\ref{eq:choquard}), we have that $(h,k,l)$ is a solution of (\ref{eq:systemderiv2}) if
\begin{equation*}
\left(\begin{array}{c}\xi(x)\\ \zeta(x)\end{array}\right)=r^{-1}\left(\begin{array}{c}h(r)\left(\begin{array}{c}1\\0\end{array}\right)\\ ik(r)\sigma^r\left(\begin{array}{c}1\\0\end{array}\right)\end{array}\right)
\end{equation*}
satisfies
\begin{equation}\label{eq:systemderiv3d1}
\left\{
\begin{array}{l}
-\triangle \xi+2m \xi-4m^3\left(\int_{\mathbb{R}^3}\frac{\left|u_0(y)\right|^2}{|x-y|}\,dy\right)\xi-8m^3\left(\int_{\mathbb{R}^3}\frac{\xi(y)u_0(y)}{|x-y|}\,dy\right)u_0=0\\[5pt]
\zeta=\frac{-i\bar{\sigma}\nabla \xi}{2m}
\end{array}
\right.
\end{equation}
and
\begin{equation}\label{eq:systemderiv3d2}
l(x)=4m\int_{\mathbb{R}^3}\frac{\xi(y)u_0(y)}{|x-y|}\,dy.
\end{equation}
It's well known that the unique solution of the first equation of (\ref{eq:systemderiv3d1}) in $H^1_{r}(\mathbb{R}^3)$ is $\xi\equiv 0$ (see \cite{Lenzmann} for more details) and that implies $\zeta\equiv l \equiv 0$. So the 
unique solution of (\ref{eq:systemderiv}) is $h\equiv k\equiv l\equiv 0$ and $D_{\varphi,\chi,\tau}(0,\phi_0)$ is one to one in $X_\varphi\times X_\chi\times X_\tau$.

Next, if we show that $S(h)$ is a compact operator, we have that $D_{\varphi,\chi,\tau}(0,\phi_0)$ is a one to one operator that can be written as a sum of an isomorphism and a compact operator and then it's an isomorphism.\\
First, we can easily see that $T(h)=\frac{1}{r^2}\left(\int_0^r\varphi_0h\,ds\right)$ is a compact operator from $X_\varphi$ on $Y_\tau$; in particular, we use the fact that $H^1_r\left(\mathbb{R}^3\right)$ is compactly embedded in $L^q\left(\mathbb{R}^3\right)$, for $2<q<6$, to prove that for any bounded sequence $\left\{h_n\right\}\subset X_\varphi$, the sequence $\left\{T(h_n)\right\}\subset Y_\tau$ contains a Cauchy subsequence.\\
Second, we have to show that the operator $\frac{h}{r}\tau_0$ from $X_\varphi$ to $L^2\left(\mathbb{R}^3\right)$ is compact. If $\left\{\frac{h_n}{r}\right\}$ is a bounded sequence in $H^1(\mathbb{R}^3)$ then $\left\{\frac{h_n}{r}\tau_0\right\}$ is precompact on $L^2_{loc}(\mathbb{R}^3)$, thanks  to compact Sobolev embedding and, since $\tau_0(r)\rightarrow 0$ when $r\rightarrow +\infty$, we can conclude that $\left\{\frac{h_n}{r}\tau_0\right\}$ is precompact on $L^2(\mathbb{R}^3)$.\\
So $S(h)$ is a compact operator from $X_\varphi$ on $Y_\varphi\times Y_\chi\times Y_\tau$ and $D_{\varphi,\chi,\tau}(0,\phi_0)$ is an isomorphism of $X_\varphi\times X_\chi\times X_\tau$ onto $Y_\varphi\times Y_\chi\times Y_\tau$.

In conclusion, we can apply the implicit function 
theorem to find that  there exists $\delta>0$ and a function $\eta\in\mathcal{C}((0,\delta),X_\varphi\times X_\chi\times X_\tau)$ such that $\eta(0)=\phi_0$ and $D(\varepsilon,\eta(\varepsilon))=0$ for $0\leq \varepsilon <\delta$.\\
\end{proof} 


\subsection*{Acknowledgment}
The author would like to thank professor Eric S\'er\'e for helpful discussions and useful comments.


\end{document}